\input ./style/arxiv-vmsta.cfg
\documentclass[numbers,compress,v1.0.1]{vmsta}
\usepackage{mathbh}
\usepackage{enumitem}
\volume{3}
\issue{2}
\pubyear{2016}
\firstpage{165}
\lastpage{179}
\doi{10.15559/16-VMSTA60}

\startlocaldefs

\allowdisplaybreaks
\newtheorem{lemma}{Lemma}
\newtheorem{cor}{Corollary}
\newtheorem{thm}{Theorem}
\theoremstyle{definition}

\newtheorem{ex}{Example}
\def\ind{\mathbh{1}}

\endlocaldefs

\begin{document}
\begin{frontmatter}

\title{Randomly stopped sums with consistently\\ varying distributions}

\author[]{\inits{E.}\fnm{Edita}\snm{Kizinevi\v{c}}}\email
{edita.kizinevic@mif.stud.vu.lt}

\author[]{\inits{J.}\fnm{Jonas}\snm{Sprindys}}\email
{jonas.sprindys@mif.stud.vu.lt}

\author[]{\inits{J.}\fnm{Jonas}\snm{\v Siaulys}\corref{cor1}}\email
{jonas.siaulys@mif.vu.lt}
\cortext[cor1]{Corresponding author.}

\address{Faculty of Mathematics and Informatics, Vilnius University,\\
Naugarduko 24, Vilnius LT-03225, Lithuania}
\markboth{E. Kizinevi\v{c} et al.}{Randomly stopped sums with consistently varying distributions}

\begin{abstract}
Let $\{\xi_1,\xi_2,\ldots\}$ be a sequence of independent random
variables, and $\eta$ be a counting random variable independent of this
sequence. We consider conditions for $\{\xi_1,\xi_2,\ldots\}$ and
$\eta$ under which the distribution function of the~random sum
$S_\eta=\xi_1+\xi_2+\cdots+\xi_\eta$ belongs to the class of
consistently varying distributions. In our consideration, the random
variables $\{\xi_1,\xi_2,\ldots\}$ are not necessarily identically
distributed.
\end{abstract}

\received{13 May 2016}
%
%
\accepted{23 June 2016}
\publishedonline{4 July 2016}
\end{frontmatter}

\section{Introduction}\label{i}

Let $\{\xi_1,\xi_2,\ldots\}$ be a sequence of independent random
variables (r.v.s) with distribution functions (d.f.s) $\{F_{\xi
_1},F_{\xi_2},\ldots\}$, and let $\eta$ be a counting r.v., that~is,
an integer-valued, nonnegative, and nondegenerate at zero~r.v. In
addition, suppose that the r.v.\ $\eta$ and r.v.s $\{\xi_1,\xi_2,\ldots
\}$ are independent. Let
$S_0=0$, $S_n=\xi_1+\xi_2+\cdots+\xi_n$ for $n\in\mathbb{N}$, and let
\begin{align*}
S_\eta=\sum_{k=1}^{\eta}
\xi_k
\end{align*}
be the randomly stopped sum of r.v.s $\{\xi_1,\xi_2,\ldots\}$.

We are interested in conditions under which the d.f. of $S_\eta$,
\begin{equation}
\label{sum} F_{S_\eta}(x)=\mathbb{P}(S_\eta\leqslant x)=\sum
_{n=0}^{\infty
}\mathbb{P}(\eta=n)
\mathbb{P}(S_n \leqslant x),
\end{equation}
belongs to the class of consistently varying distributions.

Throughout this paper, $f(x) \mathop{=} o  ( g(x)  )$ means that
$\lim_{x\to\infty} {f(x)}/{g(x)}=0$, and $f(x) \sim g(x)$ means that
$\lim_{x\to\infty} {f(x)}/{g(x)}=1$ for two vanishing (at infinity)
functions $f$ and $g$. Also, we denote the support of a counting
r.v.~$\eta$ by
\[
{\rm supp}(\eta):= \bigl\{n\in\mathbb{N}_0:\mathbb{P}(\eta=n)>0
\bigr\}.
\]

Before discussing the properties of $F_{S_\eta}$, we recall the
definitions of some classes of heavy-tailed d.f.s, where $ \overline
{F}(x) = 1-F(x)$ for all real~$x$ and a d.f. $F$.

\begin{itemize}
\item\it
A d.f. $F$ is heavy-tailed $(F \in\mathcal{H })$ if for every fixed $\delta> 0$,
\[
\lim_{x\rightarrow\infty} \overline{F}(x){\rm e}^{\delta x} = \infty.
\]
\item\it
A d.f. $F$ is long-tailed $(F\in\mathcal{L})$ if for every $y$ \textup{(}equivalently, for some $y>0$\textup{)},
\[
{\overline{F}(x+y)}\mathop{\sim} {\overline{F}(x)}.
\]
\item\it
A d.f. $F$ has a dominatedly varying tail $(F\in\mathcal{D})$
if for every fixed $y\in(0,1)$ \textup{(}equivalently, for some $y\in(0,1)$\textup{)},
\[
\limsup_{x\rightarrow\infty}\frac{\overline{F}(xy)}{\overline
{F}(x)}<\infty.
\]
\item\it
A d.f. $F$ has a consistently varying tail $(F\in\mathcal{C})$ if
\[
\lim_{y\uparrow1}\limsup_{x\to\infty} \frac{\overline{F}(xy)}{\overline {F}(x)} =
1.
\]
\item\it
A d.f. $F$ has a regularly varying tail $(F\in\mathcal{R})$ if
\[
\lim_{x\rightarrow\infty}\frac{\overline{F}(xy)}{\overline
{F}(x)} = y^{-\alpha}
\]
for some $\alpha\geqslant0$ and any fixed $y > 0$.
\item\it
A d.f. $F$ supported on the interval $[0,\infty)$ is subexponential $(F\in\mathcal{S})$ if
\begin{equation}
\label{tr}\lim_{x\rightarrow\infty}\frac{\overline{F*F}(x)}{\overline{F}(x)} = 2.
\end{equation}
If a d.f. $G$ is supported on $\mathbb{R}$, then we suppose that
$G$ is subexpo\-nential $(G\in\mathcal{S})$ if the d.f.
$F(x)=G(x)\ind_{[0,\infty)}(x)$ satisfies relation \eqref{tr}.
\end{itemize}

It is known (see, e.g., \cite{chist,eo-1984,k-1988},
and Chapters 1.4 and A3 in \cite{EKM}) that these classes satisfy the
following inclusions:
\[
\mathcal{R} \subset\mathcal{C} \subset\mathcal{L} \cap\mathcal{D} \subset
\mathcal{S} \subset\mathcal{L} \subset\mathcal{H},\qquad \mathcal {D}\subset
\mathcal{H}.
\]
These inclusions are depicted in Fig.~\ref{f1} with the class $\mathcal
{C}$ highlighted.
\begin{figure}[t]
\includegraphics{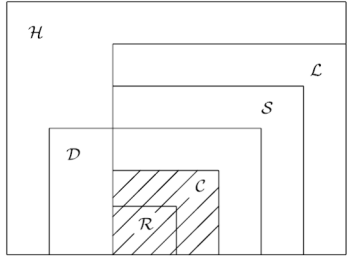}
\caption{Classes of heavy-tailed distributions.}
\label{f1}
\end{figure}

There exist many results on sufficient or necessary and sufficient
conditions in order that the d.f. of the randomly stopped sum \eqref
{sum} belongs to some heavy-tailed distribution class. Here we present
a few known results concerning the belonging of the d.f. $F_{{S}_\eta}$
to some class. The first result on subexponential distributions was proved
by Embrechts and Goldie (see Theorem 4.2 in \cite{em+gol}) and Cline
(see Theorem 2.13 in~\cite{cline}).
\begin{thm}\label{tt1}
Let $\{\xi_1,\xi_2,\ldots\}$ be independent copies of a nonnegative
r.v.\ $\xi$ with subexponential d.f. $F_\xi$. Let $\eta$ be a counting
r.v. independent of $\{\xi_1,\xi_2,\ldots\}$. If $\mathbb
{E}(1+\delta)^\eta<\infty$ for some $\delta>0$, then the d.f.
$F_{{S}_{\eta}}\in\mathcal{S}$.
\end{thm}

Similar results for the class $\mathcal{D}$ can be found in Leipus and
\v{S}iaulys \cite{ls-2012}. We present the statement of Theorem 5 from
this work.
\begin{thm}\label{tt2} Let $\{\xi_1,\xi_2,\ldots\}$ be i.i.d.
nonnegative r.v.s with common d.f. $F_\xi\in\mathcal{D}$ and finite
mean $\mathbb{E}\xi$. Let $\eta$ be a counting r.v. independent of $\{
\xi_1,\xi_2,\ldots\}$ with d.f. $F_\eta$ and finite mean $\mathbb
{E}\eta$. Then d.f. $F_{{S}_\eta}\in\mathcal{D}$ iff \, $\min\{F_\xi,
F_\eta\}\in\mathcal{D}$.
\end{thm}

We recall only that the d.f. $F$ belongs to the class $\mathcal{D}$ if
and only if the upper Matuszewska index $J^{+}_{F} < \infty$, where, by
definition,
\[
J_F^+=-\lim_{y\rightarrow\infty}\frac{1}{\log y} \log \biggl(
\liminf_{x\rightarrow\infty}\frac{\overline{F}(xy)}{\overline{F}(x)} \biggr).
\]

The random convolution closure for the class $\mathcal{L}$ was
considered, for instance, in \cite{albin,ls-2012,wat,xu}. We now present a particular statement of Theorem~1.1
from \cite{xu}.
\begin{thm}\label{tt3} Let $\{\xi_1,\xi_2,\ldots\}$ be
independent r.v.s, and $\eta$ be a counting r.v. independent of $\{\xi
_1,\xi_2,\ldots\}$ with d.f. $F_\eta$. Then the d.f.\ $F_{{S}_\eta}\in
\mathcal{L}$ if the following five conditions are satisfied:
\begin{enumerate}
\item[\rm(i)] $\mathbb{P}(\eta\geqslant\kappa)>0$ for some
    $\kappa\in \mathbb{N}$;
\item[\rm(ii)] for all $k\geqslant\kappa$, the d.f. $F_{{S}_k}$ of
    the sum ${S}_k$ is long tailed;
\item[\rm(iii)] $\displaystyle\sup_{k\geqslant1}\sup_{x\in\mathbb
    {R}}  (F_{{S}_k}(x)-F_{{S}_k}(x-1)  )\sqrt{k}<\infty$;
\item[\rm(iv)] $\displaystyle\limsup_{z\rightarrow\infty}\,\sup_{k\geqslant\kappa}\,\sup_{x\geqslant k(z-1)+z}
    \frac{\overline {F}_{{S}_k}(x-1)}{\overline{F}_{{S}_k}(x)}=1$;
\item[\rm(v)] $\overline{F}_\eta(ax)=o  (\sqrt{x}\, \overline
    {F}_{{S}_\kappa}(x)  )$ for each $a>0$.
\end{enumerate}
\end{thm}

We observe that the case of identically distributed r.v.s is considered
in Theorems~\ref{tt1} and \ref{tt2}. In Theorem \ref{tt3}, r.v.s $\{\xi
_1,\xi_2,\ldots\}$ are independent but not necessarily identically
distributed. A similar result for r.v.s having d.f.s with dominatedly
varying tails can be found in \cite{ds-2016}.
\begin{thm}[\cite{ds-2016}, Theorem 2.1]\label{tt4}
Let r.v.s $\{\xi_1,\xi_2,\ldots\}$ be nonnegative independent, not
necessarily identically distributed, and $\eta$ be a counting r.v.
independent of $\{\xi_1,\break\xi_2,\ldots\}$. Then the d.f
$F_{S_\eta}$ belongs to the class $\mathcal{D}$ if the following three
conditions are satisfied:
\begin{enumerate}
\item[\rm(i)] $F_{\xi_\kappa}\in\mathcal{D}$ for some $\kappa
    \in{\rm supp}(\eta)$,
\item[\rm(ii)] $\displaystyle \limsup_{x\rightarrow\infty} \sup_{n\geqslant\kappa}\frac{1}{n\overline{F}_{\xi_\kappa}(x)}\sum_{i=1}^{n}\overline{F}_{\xi_i}(x)<\infty$,
\item[\rm(iii)] $\mathbb{E}\eta^{p+1}<\infty$ for some $
    p>J_{F_{\xi_\kappa}}^+$.
\end{enumerate}
\end{thm}

In this work, we consider randomly stopped sums of independent and not
necessarily identically distributed r.v.s. As noted before, we restrict
ourselves on the class $\mathcal{C}$.
If r.v.s $\{\xi_1,\xi_2,\ldots\}$ are not identically distributed,
then different collections of conditions on $\{\xi_1,\xi_2,\ldots\}$
and $\eta$ imply that $F_{S_\eta}\in\mathcal{C}$. We suppose that some
r.v.s from $\{\xi_1,\xi_2, \ldots\}$ have distributions belonging to
the class $ \mathcal{C}$, and we find minimal conditions on $\{\xi_1,\xi
_2, \ldots\}$ and $\eta$ for the distribution of the randomly stopped
sum $S_\eta$ to remain in the same class. It should be noted that we
use the methods developed in \cite{ds-2016} and
\cite{dss-2016}.

The rest of the paper is organized as follows. In Section~\ref{main},
we present our main results together with two examples of randomly
stopped sums $S_\eta$ with d.f.s having consistently varying tails.
Section~\ref{lemma} is a collection of auxiliary lemmas, and the proofs
of the main results are presented in Section~\ref{proof}.

\section{Main results}\label{main}

In this section, we present three statements in which we describe the
belonging of  a randomly stopped sum to the class $\mathcal{C}$. In the
conditions of Theorem \ref{th1}, the counting r.v. $\eta$ has a finite
support. Theorem \ref{th2} describes the situation where no moment
conditions on the r.v.s $\{\xi_1,\xi_2,\ldots\}$ are required, but
there is strict requirement for $\eta$. Theorem \ref{th3} deals with
the opposite case: the r.v.s $\{\xi_1,\xi_2,\ldots\}$ should have
finite means, whereas the requirement for $\eta$ is weaker. It should
be noted that the case of real-valued r.v.s $\{\xi_1,\xi_2,\ldots\}$ is
considered in Theorems \ref{th1} and \ref{th2}, whereas Theorem \ref
{th3} deals with nonnegative r.v.s.

\begin{thm}\label{th1}
Let $\{\xi_1,\xi_2,\ldots,\xi_D\}$, $D \in\mathbb{N}$, be independent
real-valued r.v.s, and $\eta$ be a counting r.v. independent of
$\{\xi_1,\xi_2,\ldots,\xi_D\}$. Then the d.f. $F_{S_\eta }$ belongs to
the class $\mathcal{C}$ if the following conditions are satisfied:
\begin{enumerate}[label={\em(\alph*)}]
\item $\mathbb{P} (\eta\leqslant D ) = 1$,
\item $F_{\xi_{1}} \in\mathcal{C}$,
\item for each $k=2,\ldots,D$, either $F_{\xi_{k}}\in\mathcal{C}$
    or $\overline{F}_{\xi_{k}}(x) \mathop{=} o  (
    \overline{F}_{\xi_{1}}(x)  )$.
\end{enumerate}
\end{thm}

\begin{thm}\label{th2}
Let $\{\xi_1,\xi_2,\ldots\}$ be independent real-valued r.v.s, and
$\eta$ be a counting r.v. independent of $\{\xi_1,\xi_2,\ldots\}$. Then
the d.f. $F_{S_\eta}$ belongs to the class $\mathcal{C}$ if the
following conditions are satisfied:
\begin{enumerate}[label={\em(\alph*)}]
\item $F_{\xi_{1}} \in\mathcal{C}$,
\item for each  $k \geqslant2$, either $F_{\xi _{k}}
    \in\mathcal{C}$ or $\overline{F}_{\xi_{k}}(x) \mathop {=} o
     ( \overline{F}_{\xi_{1}}(x)  )$,
\item $\displaystyle \limsup_{x\to\infty} \sup_{n \geqslant1}
    \frac{1}{n\overline{F}_{\xi_1}(x)}
    \sum_{i=1}^{n}\overline{F}_{\xi_i}(x)<\infty$,
\item $\mathbb{E} \eta^{p+1} < \infty$ for some $p >
    J^{+}_{F_{\xi_1}}$.
\end{enumerate}
\end{thm}

When $\{\xi_1,\xi_2,\ldots\}$ are identically distributed with common
d.f. $F_{\xi} \in\mathcal{C}$, conditions (a), (b), and (c) of Theorem
\ref{th2} are satisfied obviously. Hence, we have the following corollary.

\begin{cor}[See also Theorem 3.4 in \cite{chen}]\label{c1}
Let $\{\xi_1,\xi_2,\ldots\}$ be i.i.d. real-valued r.v.s with d.f.
$F_{\xi} \in\mathcal{C}$, and $\eta$ be a counting r.v. independent of
$\{\xi_1,\xi_2,\ldots\}$. Then the d.f. $F_{S_\eta}$ belongs to the
class $\mathcal{C}$ if \,$\mathbb{E} \eta^{p+1} < \infty$ for some $p >
J^{+}_{F_{\xi}}$.
\end{cor}

\begin{thm}\label{th3}
Let $\{\xi_1,\xi_2,\ldots\}$ be independent nonnegative r.v.s, and
$\eta$ be a~counting r.v. independent of $\{\xi_1,\xi_2,\ldots\}$. Then
the d.f. $F_{S_\eta}$ belongs to the class $\mathcal{C}$ if the
following conditions are satisfied:
\begin{enumerate}[label={\em(\alph*)}]
\item $F_{\xi_{1}} \in\mathcal{C}$,
\item for each $k \geqslant2$, either $F_{\xi _{k}} \in\mathcal{C}$
    or $\overline{F}_{\xi_{k}}(x) \mathop {=} o  (
    \overline{F}_{\xi_{1}}(x)  )$,
\item $\mathbb{E}\xi_{1} < \infty$,
\item $\overline{F}_{\eta}(x) = o  ( \overline{F}_{\xi_{1}}(x)
     )$,
\item $\displaystyle\limsup_{x\to\infty} \sup_{n\geqslant1}
    \frac{1}{n\overline{F}_{\xi_1}(x)}
    \sum_{i=1}^{n}\overline{F}_{\xi_i}(x)<\infty$,
\item $\displaystyle\limsup_{u\to\infty} \sup_{n\geqslant1}\frac
    {1}{n} \sum_{\substack{k=1 \\ \mathbb{E} \xi_{k} \geqslant
    u}}^{n} \mathbb{E}\xi_k = 0$.
\end{enumerate}
\end{thm}

Similarly to Corollary \ref{c1}, we can formulate the following
statement. We note that, in the i.i.d. case, conditions (a), (b), (e),
and (f) of Theorem \ref{th3} are satisfied.

\begin{cor}\label{c2}
Let $\{\xi_1,\xi_2,\ldots\}$ be i.i.d. nonnegative r.v.s with common
d.f. $F_{\xi} \in\mathcal{C}$, and $\eta$ be a counting r.v.
independent of $\{\xi_1,\xi_2,\ldots\}$. Then the d.f. $F_{S_\eta}$
belongs to the class $\mathcal{C}$ under the following two conditions:
$\mathbb{E} \xi< \infty$ and $\overline{F}_{\eta}(x) \mathop{=} o (
\overline{F}_{\xi}(x)  )$.
\end{cor}

Further in this section, we present two examples of r.v.s $\{\xi_1,\xi
_2,\ldots\}$ and $\eta$ for which the random sum $F_{S_\eta}$ has a
consistently varying tail.

\begin{ex}\label{ex1}
Let $\{\xi_1,\xi_2,\ldots\}$ be independent r.v.s such that $\xi _k$
are exponentially distributed for all even $k$, that is,
\begin{eqnarray*}
\overline{F}_{\xi_k}(x)={\rm e}^{-{x}},\quad x\geqslant0,\, k\in\{
2,4,6,\ldots\},
\end{eqnarray*}
whereas, for each odd $k$, $\xi_k$ is a copy of the r.v.
\begin{equation*}
(1+\mathcal{U})\,2^{\mathcal{G}},
\end{equation*}
where $\mathcal{U}$ and $\mathcal{G}$ are independent r.v.s, $\mathcal{U}$ is
uniformly distributed on the interval $[0,1]$, and $\mathcal{G}$ is
geometrically distributed with parameter $q\in(0,1)$, that is,
\begin{equation*}
\mathbb{P}(\mathcal{G}=l)=(1-q)\,q^l,\quad l=0,1,\ldots.
\end{equation*}
In addition, let $\eta$ be a counting r.v. independent of $\{\xi_1,\xi
_2,\ldots\}$ and distributed according to the Poisson law.
\end{ex}

Theorem \ref{th2} implies that the d.f. of the randomly stopped sum
$S_\eta$ belongs to the class $\mathcal{C}$ because:
\begin{enumerate}[label={(\alph*)}]
\item $F_{\xi_{1}} \in\mathcal{C}$ due to considerations in \xch{pp.}{pages}
    122--123 of \cite{cai-tang},
\item $F_{\xi_{k}} \in\mathcal{C}$  for $k\in\{3,5,\ldots \}$, and
    $\overline{F}_{\xi_{k}}(x) \mathop{=} o  (
    \overline{F}_{\xi_{1}}(x)  )$ for $k\in\{2,4,6,\ldots\}$,
\item $\displaystyle \limsup_{x\to\infty} \sup_{n\geqslant1}\frac
    {1}{n\overline{F}_{\xi_1}(x)}\sum_{i=1}^{n} \overline{F}_{\xi
    _i}(x) \leqslant1$,
\item all moments of the r.v. $\eta$ are finite.
\end{enumerate}

Note that $\xi_1$ does not satisfy condition (c) of Theorem \ref{th3}
in the case $q\geqslant1/2$. Hence, Example \ref{ex1} describes the
situation where Theorem \ref{th2} should be used instead of Theorem \ref{th3}.

\begin{ex}\label{ex2}
Let $\{\xi_1,\xi_2,\ldots\}$ be independent r.v.s such that $\xi _k$
are distributed according to the Pareto law (with tail index
$\alpha=2$) for all odd $k$, and $\xi_k$ are exponentially distributed
(with parameter equal to $1$) for all even $k$, that is,
\begin{eqnarray*}
&&\overline{F}_{{\xi}_k}(x)=\frac{1}{x^2},\quad x\geqslant1,\ k\in\{
1,3,5,\ldots\},
\\
&&\overline{F}_{{\xi}_k}(x)={\rm e}^{-{x}},\quad x\geqslant0,\ k\in
\{ 2,4,6,\ldots\}.
\end{eqnarray*}
In addition, let $\eta$ be a counting r.v independent of $\{\xi_1,\xi
_2,\ldots\}$ that has the Zeta distribution with parameter $4$, that is,
\[
\mathbb{P}(\eta=m)=\frac{1}{\zeta(4)}\frac{1}{(m+1)^4}, \quad m\in
\mathbb{N}_0,
\]
where $\zeta$ denotes the Riemann zeta function.
\end{ex}

Theorem \ref{th3} implies that the d.f.\ of the randomly stopped sum
$S_{\eta}$ belongs to the class $\mathcal{C}$ because:
\begin{enumerate}[label={(\alph*)}]
\item $F_{\xi_{1}} \in\mathcal{C}$,
\item $F_{\xi_{k}} \in\mathcal{C}$ for $k\in\{3,5,\ldots \}$, and $
    \overline{F}_{\xi_{k}}(x) \mathop{=} o  (
    \overline{F}_{\xi_{1}}(x)  )$ for $k\in\{2,4,6,\ldots\}$,
\item $\mathbb{E}\xi_{1} = 2$,
\item $\overline{F}_{\eta}(x) = o  ( \overline{F}_{\xi_{1}}(x)
     )$,
\item $\displaystyle\limsup_{x\to\infty} \sup_{n\geqslant1}\frac
    {1}{n\overline{F}_{\xi_1}(x)}\sum_{i=1}^{n} \overline{F}_{\xi
    _i}(x) \leqslant1$,
\item $\displaystyle\max_{k\in\mathbb{N}}\mathbb{E}\xi_k=2$.
\end{enumerate}

Regarding condition (d), it should be noted that the Zeta distribution
with parameter 4 is a discrete version of Pareto distribution with tail
index 3.

Note that $\eta$ does not satisfy the condition (d) of Theorem \ref
{th2} because $J^{+}_{F_{\xi_1}} = 2$ and $\mathbb{E} \eta^3 = \infty$.
Hence, Example \ref{ex2} describes the situation where Theorem \ref
{th3} should be used instead of Theorem \ref{th2}.

\section{Auxiliary lemmas}
\label{lemma}

This section deals with several auxiliary lemmas. The first lemma is
Theorem~3.1 in \cite{chen} (see also Theorem 2.1 in \cite{wang+tang}).

\begin{lemma}\label{l1}
Let $\{X_1,X_2, \ldots X_n\}$ be independent real-valued r.v.s. If
${F}_{X_{k}} \in\mathcal{C}$ for each $k\in\{1,2,\ldots,n \}$, then
\[
\mathbb{P} \Biggl(\,\sum_{i=1}^{n}X_i>x
\Biggr)\sim\sum_{i=1}^{n}
\overline{F}_{X_i}(x).
\]
\end{lemma}

The following statement about nonnegative  subexponential distributions was proved
in Proposition 1 of \cite{egv-1979} and later generalized to a wider
distribution class in Corollary 3.19 of \cite{FKZ-2011}.

\begin{lemma}\label{l2}
Let $\{X_1, X_2,\ldots X_n\}$ be independent real-valued r.v.s. Assume
that\break $\overline{F}_{X_{i}}/
\overline{F}(x)\mathop{\rightarrow}\limits_{x\rightarrow\infty}b_i$ for some
subexponential d.f. $F$ and some constants
$b_i\geqslant0$,
$i\in\{1,2,\ldots n\}$. Then
\[
\frac{\overline{F_{X_1}*F_{X_2}*\cdots*F_{X_n}}(x)}{
\overline{F}(x)}\mathop{\rightarrow} _{x\rightarrow\infty}\sum
_{i=1}^nb_i.
\]
\end{lemma}

In the next lemma, we show in which cases the convolution
$F_{X_1}*F_{X_2}*\cdots*F_{X_n}$ belongs to the class $\mathcal{C}$.

\begin{lemma}\label{l4}
Let $\{X_1,X_2,\ldots,X_n\}, n \in\mathbb{N}$, be independent
real-valued r.v.s. Then the d.f. $F_{\varSigma_n}$ of the sum
$\varSigma _n=X_1+X_2+\cdots+ X_n$ belongs to the class $\mathcal{C}$
if the following conditions are satisfied:
\begin{enumerate}[label={\em(\alph*)}]
\item $F_{X_{1}} \in\mathcal{C}$,
\item \xch{for}{For} each $k=2,\ldots,n$, either $F_{X_{k}}
    \in\mathcal{C}$ or $\overline{F}_{X_{k}}(x) \mathop {=} o  (
    \overline{F}_{X_{1}}(x)  )$.
\end{enumerate}
\end{lemma}

\begin{proof}
Evidently, we can suppose that $n\geqslant2$. We split our proof into
two parts.

\textit{First part.} Suppose that $F_{X_k}\in\mathcal{C}$ for
all $k\in \{1,2,\ldots,n\}$. In such a case, the lemma follows from
Lemma \ref {l1} and the inequality
\begin{equation}
\label{amen} \frac{a_1+a_2+\cdots+a_m}{b_1+b_2+\cdots+b_m}\leqslant\max \biggl\{\frac
{a_1}{b_1},
\frac{a_2}{b_2},\ldots, \frac{a_m}{b_m} \biggr\}
\end{equation}
for $a_i\geqslant0$ and $b_i>0$, $i=1,2,\ldots, m$.

Namely, using the relation of Lemma \ref{l1} and estimate \eqref{amen},
we get that
\begin{align*}
\limsup_{x\rightarrow\infty}\frac{\overline{F}_{\varSigma_n}(xy)}{\overline{F}_{\varSigma_n}(x)} &{}=\limsup
_{x\rightarrow\infty} \frac{\sum_{k=1}^n\overline{F}_{X_k}(xy)}{\sum_{k=1}^n\overline{F}_{X_k}(x)}
\\
&{}\leqslant \max_{1\leqslant k\leqslant n} \limsup_{x\rightarrow\infty}
\frac{\overline{F}_{X_k}(xy)}{\overline{F}_{X_k}(x)}
\end{align*}
for arbitrary $y\in(0,1)$.

Since $F_{X_k}\in\mathcal{C}$ for each $k$, the last estimate implies
that the d.f. $F_{\varSigma_n}$ has a~consistently varying tail, as desired.

\textit{Second part.} Now suppose that
$F_{X_k}\notin\mathcal{C}$ for some of indexes $k\in\{2,3,\ldots, n\}$.
By the conditions of the lemma we have that $\overline{F}_{X_k}(x)= o (
\overline{F}_{X_1}(x) )$ for such $k$. Let
$\mathcal{K}\subset\{2,3,\ldots, n\}$ be the subset of indexes \(k\)
such that
\[
F_{X_k}\notin\mathcal{C}\quad \mbox{and} \quad \overline{F}_{X_k}(x)=
o \bigl( \overline{F}_{X_1}(x) \bigr).
\]
By Lemma \ref{l2},
\[
\overline{F}_{\widehat{\varSigma}_n}(x)\sim\overline{F}_{X_1}(x),
\]
where
\[
\widehat{\varSigma}_n=X_1+\sum
_{k\in\mathcal{K}}X_k.
\]
Hence,
\begin{equation}
\label{amen2} \limsup_{x\rightarrow\infty}\frac{\overline{F}_{\widehat{\varSigma
}_n}(xy)}{\overline{F}_{\widehat{\varSigma}_n}(x)} =\limsup
_{x\rightarrow\infty} \frac{\overline
{F}_{X_1}(xy)}{\overline{F}_{X_1}(x)}
\end{equation}
for every $y\in(0,1)$.

Equality \eqref{amen2} implies immediately that the d.f. $F_{\widehat
{\varSigma}_n}$ belongs to the class $\mathcal{C}$. Therefore, the d.f.
$F_{\varSigma_n}$ also belongs to the class $\mathcal{C}$ according to the
first part of the proof because
\[
\varSigma_n=\widehat{\varSigma}_n+\sum
_{k\notin\mathcal{K}}X_k
\]
and $F_{X_k}\in\mathcal{C}$ for each $k\notin\mathcal{K}$. The lemma is proved.
\end{proof}

The following two statements about dominatedly varying distributions
are\break Lemma~3.2 and Lemma 3.3 in \cite{ds-2016}. Since any consistently
varying distribution is also dominatingly varying, these statements
will be useful in the proofs of our main results concerning the class
$\mathcal{C}$.

\begin{lemma}\label{l5}
Let $\{X_1,X_2,\ldots\}$ be independent real-valued r.v.s, and
$F_{X_{\nu}} \in\mathcal{D}$ for some $\nu\geqslant1$. Suppose, in
addition, that
\begin{equation*}
\limsup_{x\to\infty} \sup_{n\geqslant\nu}\frac
{1}{n\overline{F}_{X_{\nu}}(x)}
\sum_{i=1}^{n} \overline
{F}_{X_i}(x)<\infty.
\end{equation*}
Then, for each $p > J^{+}_{F_{X_{\nu}}}$, there exists a positive
constant $c_1$ such that
\begin{equation}
\overline{F}_{S_{n}}(x) \leqslant c_1 n^{p+1}
\overline{F}_{X_\nu}(x)
\end{equation}
for all $n \geqslant\nu$ and $x \geqslant0$.
\end{lemma}

In fact, Lemma \ref{l5} is proved in \cite{ds-2016} for nonnegative
r.v.s. However, the lemma remains valid for real-valued r.v.s. To see
this, it suffices to observe that
$
\mathbb{P}(X_1+X_2+\cdots+X_n>x)\leqslant\mathbb{P}(X_1^++X_2^+\cdots+X_n^+>x)$
and $ \mathbb{P}(X_k>x)=\mathbb{P}(X_k^+>x) $, where $n\in\mathbb{N}$,
$k\in\{1,2,\ldots,n\}$, $x\geqslant0$, and $a^+$ denotes the positive
part of $a$.

\begin{lemma}\label{l6}
Let $\{X_1,X_2,\ldots\}$ be independent real-valued r.v.s, and
$F_{X_\nu}\in\mathcal{D}$ for some $\nu\geqslant1$. Let, in addition,
\begin{eqnarray*}
\label{d1} &&\lim_{u\rightarrow\infty}\sup_{n\geqslant\nu}
\frac
{1}{n}\sum_{k=1}^{n}\mathbb{E} \bigl(|X_k| \ind_{\{X_k\leqslant-u\}}\bigr)=0,
\\
&& \limsup_{x\to\infty} \sup_{n\geqslant\nu}
\frac
{1}{n\overline{F}_{X_{\nu}}(x)}\sum_{i=1}^{n} \overline
{F}_{X_i}(x)<\infty,
\end{eqnarray*}
and $\mathbb{E}X_k=\mathbb{E}X_k^+-\mathbb{E}X_k^-=0$ for $k\in\mathbb
{N}$. Then, for each $\gamma>0$, there exists a~positive constant
$c_2=c_2(\gamma)$ such that
\[
\mathbb{P}(S_n>x)\leqslant c_2 n\overline{F}_{X_\nu}(x)
\]
for all $x\geqslant\gamma n$ and all $n\geqslant\nu$.
\end{lemma}\vfill

\section{Proofs of the main results}\label{proof}

\begin{proof}[Proof of Theorem \ref{th1}]
It suffices to prove that
\begin{equation}
\label{th1.1} \limsup_{y\uparrow1}\limsup_{x\to\infty}
\frac{\overline{F}_{S_{\eta
}}(xy)}{\overline{F}_{S_{\eta}}(x)} \leqslant1.
\end{equation}
According to estimate \eqref{amen}, for $x > 0$ and $y \in(0,1)$, we have
\begin{equation*}
\begin{aligned} \frac{\overline{F}_{S_{\eta}}(xy)}{\overline{F}_{S_{\eta}}(x)} & = \frac{\sum_{\substack{n=1 \\ n \in{\rm supp}(\eta)}}^{D}
\mathbb{P} (S_n>xy ) \mathbb{P} (\eta=n )}{\sum_{\substack{n=1 \\ n \in{\rm supp}(\eta)}}^{D} \mathbb{P}
(S_n>x ) \mathbb{P} (\eta=n )} \leqslant
\max_{\substack
{1 \leqslant n \leqslant D \\ n \in{\rm supp}(\eta)}} \frac{\mathbb
{P} (S_n>xy )}{\mathbb{P} (S_n>x )}. \end{aligned} %
\end{equation*}
Hence, by Lemma \ref{l4},
\begin{equation*}
\begin{aligned} \limsup_{y\uparrow1}\limsup
_{x\to\infty} \frac{\overline{F}_{S_{\eta
}}(xy)}{\overline{F}_{S_{\eta}}(x)} & \leqslant\limsup
_{y\uparrow1}\limsup_{x\to\infty} \max_{\substack{1
\leqslant n \leqslant D \\ n \in{\rm supp}(\eta)}}
\frac{\overline
{F}_{S_{n}}(xy)}{\overline{F}_{S_{n}}(x)}
\\
& \leqslant\max_{\substack{1 \leqslant n \leqslant D \\ n \in{\rm
supp}(\eta)}} \limsup_{y\uparrow1}\limsup
_{x\to\infty} \frac{\overline
{F}_{S_{n}}(xy)}{\overline{F}_{S_{n}}(x)} = 1, \end{aligned} %
\end{equation*}
which implies the desired estimate \eqref{th1.1}. The theorem is
proved.
\end{proof}

\begin{proof}[Proof of Theorem \ref{th2}]
As in Theorem \ref{th1}, it
suffices to prove inequality \eqref{th1.1}.
For all $K \in\mathbb{N}$ and $x>0$, we have
\begin{equation*}
\mathbb{P} (S_{\eta}>x )= \Biggl(\sum_{n=1}^{K}+
\sum_{n=K+1}^{\infty} \Biggr)\mathbb{P}
(S_{n}>x )\mathbb{P} (\eta=n ).
\end{equation*}
Therefore, for $x > 0$ and $y \in(0,1)$, we have
\begin{align}
\frac{\mathbb{P} (S_{\eta}>xy )}{\mathbb{P} (S_{\eta}>x )}
&{}= \frac{\sum_{n=1}^{K} \mathbb{P}(S_{n}>xy ) \mathbb{P} (\eta=n )}{\mathbb{P} (S_{\eta}>x )}\nonumber\\
&\quad{} +\ \frac{\sum_{n=K+1}^{\infty} \mathbb{P} (S_{n}>xy )\mathbb{P} (\eta=n )}{\mathbb{P} (S_{\eta}>x)}\nonumber\\
&{} =: \mathcal{J}_{1} + \mathcal{J}_{2}.\label{th2.1}
\end{align}

The random variable $\eta$ is not degenerate at zero, so there exists
$a \in\mathbb{N}$ such that $\mathbb{P} (\eta=a )>0$.
If $K \geqslant a$, then using inequality \eqref{amen}, we get
\begin{equation*}
\mathcal{J}_{1} \leqslant\frac{\sum_{\substack{n=1 \\ n \in
{\rm supp}(\eta)}}^{K} \mathbb{P} (S_n>xy ) \mathbb{P}
(\eta=n )}{\sum_{\substack{n=1 \\ n \in{\rm supp}(\eta
)}}^{K} \mathbb{P} (S_n>x ) \mathbb{P} (\eta=n )} \leqslant\max
_{\substack{1 \leqslant n \leqslant K \\ n \in{\rm
supp}(\eta)}}\frac{\mathbb{P} (S_n>xy )}{\mathbb{P}
(S_n>x )}.
\end{equation*}
Similarly as in the proof of Theorem \ref{th1}, it follows that
\begin{equation}
\label{th2.2} \limsup_{y\uparrow1}\limsup_{x\to\infty}
\mathcal{J}_{1} \leqslant\max_{\substack{1 \leqslant n \leqslant K \\ n \in{\rm supp}(\eta)}} \limsup
_{y\uparrow1}\limsup_{x\to\infty} \frac{\overline
{F}_{S_{n}}(xy)}{\overline{F}_{S_{n}}(x)} = 1.
\end{equation}
Since $\mathcal{C} \subset \mathcal{D}$, we can use Lemma \ref{l5} for
the numerator of $\mathcal{J}_{2}$ to obtain
\begin{equation*}
\sum_{n=K+1}^{\infty} \mathbb{P}
(S_{n}>xy ) \mathbb{P} (\eta=n ) \leqslant c_3
\overline{F}_{\xi_{1}}(xy)\sum_{n=K+1}^{\infty}n^{p+1}
\mathbb{P} (\eta=n )
\end{equation*}
with some positive constant $c_3$.
For the denominator of $\mathcal{J}_{2}$, we have that
\begin{align*}
\mathbb{P} (S_{\eta}>x )
&{}
=
\sum_{n=1}^{\infty}
\mathbb{P} (S_n>x )\mathbb {P}(\eta=n)
\nonumber
\\
&{}
\geqslant
\mathbb{P}(S_a>x) \mathbb{P} (\eta=a ).
\end{align*}

The conditions of the theorem imply that
\[
S_a=\xi_1+\sum_{k\in\mathcal{K}_a}
\xi_k+\sum_{k\notin
\mathcal{K}_a}\xi_k,
\]
where $\mathcal{K}_a= \{ k\in\{2,\ldots, a\} :F_{\xi_k}\notin\mathcal
{C}, \overline{F}_{\xi_k}(x)= o (\,\overline{F}_{\xi_1}(x) ) \}$.

By Lemma \ref{l2}
\[
\overline{F}_{\widehat{S}_a}(x)/\overline{F}_{\xi_1}(x)\mathop {
\rightarrow}_{x\rightarrow\infty}1,
\]
where ${F}_{\widehat{S}_a}$ is the d.f. of the sum
\[
\widehat{S}_a=\xi_1+\sum_{k\in\mathcal{K}_a}
\xi_k
\]
In addition, by Lemma \ref{l4} we have that the d.f. $F_{\widehat{S}_a}$
belongs to the class $\mathcal{C}$.

If $k\notin\mathcal{K}_a$, then $F_{\xi_k}\in\mathcal{C}$ by the
conditions of the theorem. This fact and Lemma~\ref{l1} imply that
\[
\liminf_{x\rightarrow\infty}\frac{\mathbb{P}(S_a>x)}{\overline
{F}_{\xi_1}(x)}\geqslant1+\sum
_{k\notin\mathcal{K}_a} \liminf_{x\rightarrow\infty}\frac{\overline{F}_{\xi_k}(x)}{\overline
{F}_{\xi_1}(x)}.
\]
Hence,
\begin{equation}
\label{th2.3} \mathbb{P}(S_\eta>x)\geqslant\frac{1}{2}
\overline{F}_{\xi_1}(x)\mathbb {P}(\eta=a)
\end{equation}
for $x$ sufficiently large. Therefore,
\begin{align}
&\limsup_{y\uparrow1}\limsup_{x\to\infty}\mathcal {J}_{2}\nonumber\\
&\quad{} \leqslant
\frac{2\,c_3}{\mathbb{P} (\eta=a )} \biggl(\limsup_{y\uparrow1}
\limsup_{x\to\infty}
\frac{\overline{F}_{\xi_{1}}(xy)}{\overline{F}_{\xi_{1}}(x)} \biggr)
\sum_{n=K+1}^{\infty}n^{p+1} \mathbb{P} (\eta=n ).\label{th2.4}
\end{align}
Estimates \eqref{th2.1}, \eqref{th2.2}, and \eqref{th2.4} imply that
\begin{equation*}
\limsup_{y\uparrow1}\limsup_{x\to\infty}
\frac{\mathbb{P} (S_{\eta
}>xy )}{\mathbb{P} (S_{\eta}>x )} \leqslant1 + \frac{2\,
c_3}{\mathbb{P} (\eta=a )}\mathbb{E} {\eta}^{p+1}
\ind_{\{\eta
>K\}}
\end{equation*}
for arbitrary $K \geqslant a$.

Letting $K$ tend to infinity, we get the desired estimate \eqref{th1.1}
due to condition (d). The theorem is proved.
\end{proof}

\begin{proof}[Proof of Theorem \ref{th3}]
Once again, it suffices to prove
inequality \eqref{th1.1}.

By condition (e) we have that there exist two positive constants $c_4$
and $c_5$ such that
\begin{equation*}
\sum_{i=1}^{n}\overline{F}_{\xi_{i}}(x)
\leqslant c_{5}n\overline {F}_{\xi_{1}}(x), \quad x \geqslant c_4,\ n \in\mathbb{N}.
\end{equation*}
Therefore,
\begin{equation}
\label{th3.1} \mathbb{E}S_n = \sum_{j=1}^{n}
\mathbb{E}\xi_{j} = \sum_{j=1}^{n}
\Biggl(\, \int_{0}^{c_4} + \int
_{c_4}^{\infty}\, \Biggr) \overline{F}_{\xi_{j}}(u)du
\leqslant c_4 n + c_5 n \mathbb{E} \xi_1 =:
c_6 n
\end{equation}
for a positive constant $c_6$ and all $n \in\mathbb{N}$.

If $K \in\mathbb{N}$ and $x > 4Kc_6$, then we have
\begin{align*}
\mathbb{P} ( S_{\eta} > x ) &{}= \mathbb{P} ( S_{\eta} > x,\eta\leqslant K )
\\
&\quad{} + \mathbb{P} \biggl( S_{\eta} > x, K < \eta\leqslant \frac{x}{4c_6} \biggr)
\\
&\quad{}+ \mathbb{P} \biggl( S_{\eta} > x, \eta> \frac{x}{4c_6} \biggr).
\end{align*}
Therefore,
\begin{align}
\frac
{\mathbb{P}  ( S_{\eta} > xy  )}
{\mathbb{P}  ( S_{\eta} > x  )}
&{} =
\frac
{\mathbb{P}  ( S_{\eta} > xy, \eta\leqslant K)}
{\mathbb{P}  ( S_{\eta} > x  )}
\nonumber
\\
&\quad{} +
\frac
{\mathbb{P}  \big( S_{\eta} > xy, K < \eta\leqslant\frac{xy}{4c_6}\big)}
{\mathbb{P}  ( S_{\eta} > x  )}
\nonumber
\\
&\quad{} +
\frac
{\mathbb{P}  \big( S_{\eta} > xy, \eta> \frac{xy}{4c_6} \big)}
{\mathbb{P}  ( S_{\eta} > x  )}
\nonumber
\\
&{} =: \mathcal{I}_1 + \mathcal{I}_2 + \mathcal{I}_3
\label{th3.2}
\end{align}
if $xy > 4Kc_6,\ x > 0$, and $y \in(0,1)$.

The random variable $\eta$ is not degenerate at zero, so $\mathbb
{P} (\eta=a )>0$ for some $a \in\mathbb{N}$.
If $K \geqslant a$, then
\begin{equation}
\label{th3.3} \limsup_{y\uparrow1}\limsup_{x\to\infty}
\mathcal{I}_{1} \leqslant1
\end{equation}
similarly to estimate \eqref{th2.2} in Theorem \ref{th2}.

For the numerator of $\mathcal{I}_{2}$, we have
\begin{align}
\mathcal{I}_{2,1} :={}&
\mathbb{P} \biggl(S_{\eta} > xy, K < \eta \leqslant\frac{xy}{4c_6} \biggr)
\nonumber
\\
={}& \sum_{K<n\leqslant\frac{xy}{4c_6} }\mathbb{P} \Biggl(\sum_{i=1}^{n}( \xi_i-\mathbb {E}\xi_i )> xy- \sum_{j=1}^{n}\mathbb{E}\xi_j \Biggr) \mathbb {P}(\eta=n)
\nonumber
\\
\leqslant{}&
\sum_{K<n\leqslant\frac{xy}{4 c_{6}}}\mathbb{P} \Biggl(\sum_{i=1}^{n} (\xi_i-\mathbb{E}\xi_i ) > \frac{3}{4} xy \Biggr) \mathbb{P}(\eta=n)
\label{th3.4}
\end{align}
by inequality \eqref{th3.1}.\goodbreak

The random variables $\xi_1-\mathbb{E}\xi_1$,
$\xi_2-\mathbb{E}\xi_2,\dots$ satisfy the conditions of Lemma~\ref{l6}.
Namely, $\mathbb {E}(\xi_k-\mathbb{E}\xi_k)=0$ for $k \in\mathbb{N}$
and $F_{\xi _1-\mathbb{E}\xi_1} \in\mathcal{C} \subset\mathcal{D}$
obviously. In addition,
\begin{align*}
&\limsup_{x\rightarrow\infty}\,\sup_{n\geqslant1}\frac{1}{n\,\mathbb{P}(\xi_1-\mathbb{E}\xi_1>x)}\sum_{k=1}^{n}\mathbb{P}(\xi_i-\mathbb{E}\xi_i>x)< \infty
\end{align*}
by conditions (a), (c) and (e). Finally,
\begin{align*}
&\limsup_{u\rightarrow\infty}\,\sup_{n\geqslant1}
\frac{1}{n}\sum_{k=1}^{n}\mathbb{E}(|\xi_k-\mathbb{E} \xi _k|\ind_{\{\xi_k-\mathbb{E}\xi_k \leqslant-u\}})
\\
&\quad{}=
\limsup_{u\rightarrow\infty}\,\sup_{n\geqslant 1}
\frac{1}{n}
\sum_{k=1}^{n}\mathbb{E}
\bigl((\mathbb{E} \xi_k-\xi_k)\ind_{\{\xi_k-\mathbb{E}\xi_k \leqslant-u\}}\bigr)
\\
&{}\leqslant
\limsup_{u\rightarrow\infty}\sup_{n\geqslant 1}
\frac{1}{n}\sum_{\substack{1\leqslant k\leqslant n \\ \mathbb{E}\xi_k\geqslant u }}\mathbb{E}
\xi_k=0
\end{align*}
because of condition (f). So,
applying the estimate of Lemma \ref{l6} to \eqref{th3.4}, we get
\begin{align*}
\mathcal{I}_{2,1}
&{}\leqslant c_7 \sum _{K<n\leqslant\frac{xy}{4 c_6}}n\overline{F}_{\xi_1} \biggl(
\frac{3}{4}xy+\mathbb{E}\xi_1 \biggr)\mathbb{P}(\eta=n)
\\
&{}\leqslant\ c_7 \overline{F}_{\xi_1} \Bigl(\,\frac{3}{4}xy \Bigr) \mathbb{E}\eta\ind_{\{\eta>K\}}
\end{align*}
with a positive constant $c_7$.
For the denominator of $\mathcal{I}_{2}$, we can use the inequality
\begin{align}
\mathbb{P} (S_{\eta}>x )
&{}=
\sum_{n=1}^{\infty} \mathbb{P} (S_n>x )\mathbb{P}(\eta=n)
\nonumber
\\
&{}\geqslant
\sum_{n=1}^{\infty}\mathbb{P} (\xi_1>x )\mathbb{P}(\eta=n)
\nonumber
\\
&{}\geqslant \overline{F}_{\xi_{1}}(x) \mathbb{P} (\eta=a )
\label{am1}
\end{align}
since the r.v.s $\{\xi_1,\xi_2,\ldots\}$ are nonnegative by assumption.
Hence,
\begin{equation*}
\mathcal{I}_{2} \leqslant\frac{c_7}{\mathbb{P}(\eta= a)} \mathbb {E}\eta
\ind_{\{\eta>K\}} \frac{\overline{F}_{\xi_1} \big(\frac
{3}{4}xy\big)}{\overline{F}_{\xi_1} (x )}.
\end{equation*}

If $y \in(1/2,1)$, then the last estimate implies that
\begin{equation}
\label{th3.5} \limsup_{x\rightarrow\infty} \mathcal{I}_{2}
\leqslant \frac
{c_7}{\mathbb{P}(\eta= a)} \mathbb{E}\eta\ind_{\{\eta>K\}} \limsup
_{x\rightarrow\infty}\ \frac{\overline{F}_{\xi_1} \big(\frac
{3}{8}x \big)}{\overline{F}_{\xi_1} (x )} \leqslant c_8 \mathbb{E}
\eta\ind_{\{\eta>K\}}
\end{equation}
with some positive constant $c_8$ because $F_{\xi_1} \in\mathcal{C}
\subset\mathcal{D}$.

Using inequality \eqref{am1} again, we obtain
\begin{equation*}
\mathcal{I}_{3} \leqslant\frac{\mathbb{P} \big(\eta> \frac{xy}{4c_6}
\big)}{\mathbb{P} ( S_{\eta} > x )} \leqslant
\frac{1}{\mathbb{P} ( \eta=
a )} \frac{\overline{F}_{\eta} \big(\frac{xy}{4c_6} \big)}{\overline
{F}_{\xi_1} \big(\frac{xy}{4c_6} \big)} \frac{\overline{F}_{\xi_1} \big(\frac{xy}{4c_6} \big)}{\overline{F}_{\xi_1} (x )}.
\end{equation*}
Therefore, for $y \in(1/2,1)$, we get
\begin{equation}
\label{th3.6} \limsup_{x\rightarrow\infty} \mathcal{I}_{3}
\leqslant \frac
{1}{\mathbb{P} ( \eta= a )} \limsup_{x\rightarrow\infty} \frac
{\overline{F}_{\eta} \big(\frac{xy}{4c_6} \big)}{\overline{F}_{\xi_1} \big(\frac{xy}{4c_6}\big)}
\limsup_{x\rightarrow\infty} \frac{\overline{F}_{\xi_1} \big(\frac{xy}{4c_6} \big)}{\overline
{F}_{\xi_1} (x )} = 0
\end{equation}
by condition (d).\goodbreak

Estimates \eqref{th3.2}, \eqref{th3.3}, \eqref{th3.5}, and \eqref
{th3.6} imply that
\begin{equation*}
\limsup_{y\uparrow1}\limsup_{x\to\infty}\frac{\mathbb
{P} (S_{\eta}>xy )}{\mathbb{P} (S_{\eta}>x
)}
\leqslant1 + c_8 \mathbb{E}\eta\ind_{ \lbrace\eta> K
\rbrace}
\end{equation*}
for $K \geqslant a$.

Letting $K$ tend to infinity, we get the desired estimate \eqref{th1.1}
because $\mathbb{E}\eta< \infty$ by conditions (c) and (d). The
theorem is proved.
\end{proof}




\end{document}